\documentclass{amsart}
\usepackage{amsfonts}
\usepackage{amsmath,amssymb}
\usepackage{amsthm}
\usepackage{amscd}
\usepackage{graphics}
\usepackage{graphicx}

\theoremstyle{remark}{
\newtheorem{Def}{{\rm Definition}}

\newtheorem{Rem}{{\rm Remark}}
\newtheorem{Prob}{{\rm Problem}}

}
\theoremstyle{plain}
{

\newtheorem{Prop}{Proposition}
\newtheorem{Thm}{Theorem}

}

\begin{document}
\title[Representation of Reeb spaces via simplified graphs]{Representations of Reeb spaces via simplified graphs and examples}
\author{Naoki kitazawa}
\keywords{Smooth, real analytic, or real algebraic (real polynomial) functions and maps. Reeb spaces. Cell complexes. 1-dimensional cell complexes whose edges are oriented. Peano continua. (Di)graphs. Reeb (di)graphs. \\
\indent {\it \textup{2020} Mathematics Subject Classification}: Primary~26C05, 26E05, 54C30, 54F15, 57R45, 58C05.}

\address{Osaka Central Advanced Mathematical Institute (OCAMI) \\
3-3-138 Sugimoto, Sumiyoshi-ku Osaka 558-8585
TEL: +81-6-6605-3103
}
\email{naokikitazawa.formath@gmail.com}
\urladdr{https://naokikitazawa.github.io/NaokiKitazawa.html}
\maketitle
\begin{abstract}
{\it Reeb spaces} of continuous real-valued functions on topological spaces are fundamental and strong tools in investigating the spaces. The {\it Reeb space} is a natural quotient space of the space of the domain represented by connected components of its level sets. They have appeared in theory of Morse functions in the last century and as important topological objects. They are shown to be graphs for tame functions on (compact) manifolds such as Morse(-Bott) functions and naturally generalized ones. Related general theory develops actively, recently, mainly by Gelbukh and Saeki. For nice Haudorff spaces and continuous functions there, they are "$0$- or $1$-dimensional".

We study Reeb spaces which are not CW complexes and study their representations by graphs and nice examples. Reconstructing nice smooth functions with given Reeb graphs is of related studies, pioneered by Sharko and followed by Masumoto, Michalak, Saeki, and so on. The author has also contributed to it.

\end{abstract}
\section{Introduction.}
\label{sec:1}

The {\it Reeb spaces} $R_c$ of continuous real-valued functions $c:X \rightarrow \mathbb{R}$ on topological spaces $X$ are fundamental and strong tools in investigating the spaces. They have appeared in theory of Morse functions in the last century. \cite{reeb} is one of related pioneering papers. More rigorously, we can consider the equivalence relation ${\sim}_c$ on $X$ and $R_c:=X/{\sim c}$. For $x_1,x_2 \in X$, $x_1 {\sim}_c x_2$ if and only if they are in a same connected component of a {\it level set} $c^{-1}(y)$. We define the notion of {\it level set} of the function in this way. A {\it contour} of $c$ is a connected component of a level set of of $c$. We have the quotient map $q_c:X \rightarrow R_c$ and the unique continuous function $\bar{c}:R_c \rightarrow \mathbb{R}$ with $c=\bar{c} \circ q_c$.

They have been shown to be graphs for tame functions on (compact) manifolds such as Morse(-Bott) functions and naturally generalized ones, according to \cite{izar, martinezalfaromezasarmientooliveira}, and general studies by \cite{saeki1, saeki2}. 

We give exposition on singularity of smooth maps. A {\it singular} point of a differentiable map is a point of the manifold of the domain where the rank of the differential is smaller than both the dimension of the manifold of the domain and that of the target. We use "{\it critical}" instead of "singular" in the case where the manifold of the target is $1$-dimensional. 
Let $S(c)$ denote the set of all singular points of $c$. 

For the graphs above, a point $v$ is a vertex if ${q_c}^{-1}(v)$ contains some critical point of $c$. We say that such a contour of $c$ is {\it critical}. A contour of $c$ which is not critical is {\it regular}. $R_c$ is the {\it Reeb graph} of $c$. We can orient each edge $e$ incident to $v_1$ and $v_2$ in such a way that $e$ departs from $v_1$ and enters $v_2$ if and only if $\bar{c}(v_1)<\bar{c}(v_2)$ holds. This is a digraph (the {\it Reeb digraph} of $c$).

By general theory such as \cite{gelbukh1, gelbukh2, gelbukh3, gelbukh4, saeki1, saeki2}, it may not be a graph or a CW complex. For smooth real-valued functions on smooth closed and connected manifolds and more generally, continuous real-valued functions on compact, connected, locally connected and spaces ({\it Peano continua}) which are metrizable, the Reeb spaces are also ($1$-dimensional and metrizable) Peano continua. In \cite{gelbukh3, gelbukh4}, for certain general real-valued continuous functions on topological spaces, it is studied whether the Reeb spaces are Hausdorff. 

Related to this, the author has been interested in reconstructing nice smooth functions with given Reeb graphs. Such a problem has been first pioneered by Sharko in \cite{sharko} and followed by Masumoto, Michalak, Saeki, and so on. For this, see \cite{masumotosaeki, michalak} for example. 
These studies are essentially on reconstruction of nice smooth functions on closed surfaces which are Morse or certain elementary polynomials around their critical points. The author has also contributed to it with their higher dimensional cases and respecting shapes of level sets in addition. See papers such as \cite{kitazawa1, kitazawa4} and for non-compact (proper) cases see the paper \cite{kitazawa2} for example. \cite{kitazawa5} is also a related preprint of the author.

As a kind of interdisciplinary studies, the author has been interested in explicit construction of not only smooth, but also real algebraic, or real analytic functions. In the closed manifold cases, the author has systematically constructed and shown realizations of given Reeb (di)graphs of certain classes by constructing real algebraic functions. The author has generalized the so-called natural height function of the $m$-dimensional unit sphere $S^m:=\{x:=(x_1 \cdots x_{m+1}) \in {\mathbb{R}}^{m+1} \mid {\Sigma}_{j=1}^{m+1} {x_j}^2=1 \}$ in the ($m+1$)-dimensional Euclidean space ${\mathbb{R}}^{m+1}$ with $m \geq 1$, defined as the restriction of the projection to the first component by the projection ${\pi}_{m+1,1}(x):=x_1$. We define the canonical projection ${\pi}_{k,k_1}:{\mathbb{R}^{k}} \rightarrow {\mathbb{R}}^{k_1}$ by ${\pi}_{k,k_1}(x):=x_1$ ($x:=(x_1.x_2) \in {\mathbb{R}}^{k_1} \times {\mathbb{R}}^{k_2}={\mathbb{R}}^{k}$) generally. See the pioneering paper \cite{kitazawa3} and see also \cite{kitazawa6, kitazawa7}.

We explain some fundamental terminologies and notation in addition. We use $c {\mid}_Z$ for the restriction of a map $c:X \rightarrow Y$ to a subset $Z\subset X$. The canonical projection of the unit sphere $S^m$ is ${\pi}_{m+1,n} {\mid}_{S^m}$ and generalizing the natural height of $S^m$. Its image is the $n$-dimensional unit disk $D^n \subset {\mathbb{R}}^n$. A map $c:X \rightarrow Y$ between topological spaces is {\it proper} if $c^{-1}(K)$ is compact for each compact subset $K \subset Y$ and a {\it non-proper} map between topological spaces is a map which is not proper.

We go back to our main issue. We discuss the following, which are also questioned in \cite{kitazawa8, kitazawa9, kitazawa10, kitazawa11, kitazawa12} for example.
\begin{Prob}
\label{prob:1}
How should realization of Reeb spaces which are not homeomorphic to finite graph by nice smooth functions be formulated?
\end{Prob}
\begin{Prob}
\label{prob:2}
Can we explain examples for Problem \ref{prob:1}?
\end{Prob}
In the previous preprints \cite{kitazawa9, kitazawa10, kitazawa11, kitazawa12}, respecting the pioneering result \cite[Theorem 1]{kitazawa8}, we consider the following for example. Hereafter, for a topological space $X$ and its subspace $Y \subset X$, ${\overline{Y}}^X \subset X$ is the closure of $Y$ in $X$. 

Let $c_i:\mathbb{R} \rightarrow \mathbb{R}$ ($i=1,2$) be two real analytic functions such that $c_1(x)<c_2(x)$ for $x \in \mathbb{R}$.
We consider the region $D_{c_1,c_2}:=\{(x_1,x_2) \mid c_1(x_2)<x_1<c_2(x_2)\}$ and define the set $X_{m,c_1,c_2}:=\{(x_1,x_2,{(y_j)}_{j=1}^{m-1}) \mid (x_1-c_1(x_2))(c_2(x_2)-x_1)-{\Sigma}_{j=1}^{m-1} {y_j}^2=0 \} \subset {\mathbb{R}}^{m+1}$ for each integer $m>1$. We also have $X_{m,c_1,c_2}=\{(x_1,x_2,{(y_j)}_{j=1}^{m-1}) \in {\overline{D_{c_1,c_2}}}^{{\mathbb{R}}^2} \times {\mathbb{R}}^{m-1} \mid (x_1-c_1(x_2))(c_2(x_2)-x_1)-{\Sigma}_{j=1}^{m-1} {y_j}^2=0 \} \subset {\mathbb{R}}^{m+1}$. This is, by implicit function theorem, an $m$-dimensional smooth submanifold of ${\mathbb{R}}^{m+1}$ with no boundary. More precisely, the value of the partial derivative of $(x_1-c_1(x_2))(c_2(x_2)-x_1)-{\Sigma}_{j=1}^{m-1} {y_j}^2$ by some $y_j$ is not $0$ in the case $(x_1,x_2) \in D_{c_1,c_2}$ and that of $(x_1-c_1(x_2))(c_2(x_2)-x_1)-{\Sigma}_{j=1}^{m-1} {y_j}^2$ by $x_1$ is not $0$ in the case $(x_1,x_2) \in \overline{D_{c_1,c_2}}-D_{c_1,c_2}$. For this, see \cite{kitazawa4}, where the real algebraic or polynomial case is considered and a kind of pioneering studies on real algebraic construction of functions and maps. We do not assume knowledge or arguments of \cite{kitazawa4} and this is no problem.

We call $(c_1,c_2,D_{c_1,c_2},X_{m,c_1,c_2})$ the {\it $c_1<c_2$-zero-set}. We use this name first in the present paper.

The author has investigated topologies of Reeb spaces $R_{{\pi}_{m+1,1} {\mid}_{X_{m,c_1,c_2}}}$.

For example, Proposition \ref{prop:1} is of fundamental propositions. It is also proven in the existing preprints of the auhtor. This can be also immediately shown.
\begin{Prop}
\label{prop:1}
 A point of $X_{m,c_1,c_2}$ is a critical point of ${\pi}_{m+1,1} {\mid}_{X_{m,c_1,c_2}}$ if and only if it is a point mapped to a point of the form $(c_j(p),p)$ {\rm (}$p \in S(c_j)${\rm )} by ${\pi}_{m+1,2} {\mid}_{X_{m,c_1,c_2}}$. Furthermore, to each of such points, exactly one point of $X_{m,c_1,c_2}$ is mapped by ${\pi}_{m+1,2}$. 
\end{Prop}

The author has also posed constraints that $c_1(S(c_1)) \bigcup c_2(S(c_2))$ is a discrete and closed set outside some discrete and closed set $Z_{c_1,c_2} \subset \mathbb{R}$.
Under this, they are also shown to be ($1$-dimensional and metrizable) Peano continua which are also regarded as $1$-dimenional cell complexes, in \cite{kitazawa11}. Through these studies, we also give examples by explicit real analytic functions $c_1$ and $c_2$. See \cite{kitazawa10, kitazawa11} mainly, where we do not assume related knowledge or nontrivial arguments. This is a kind of answers to Problems \ref{prob:1} and \ref{prob:2}. 

We are also reviewing existing studies on real analytic reconstruction of ({\it non-proper}) functions with topological information and combinatorial one. More precisely, this is essentially studied by the author, first in the preprint \cite{kitazawa8} (\cite[Theorem 1]{kitazawa8}). Of course, we do not assume related knowledge or arguments.

In the present paper, we discuss Problems \ref{prob:1} and \ref{prob:2} and present answers in certain new ways. In the next section, first, we explain cell complexes and $1$-dimensional cases with natural orientations induced by continuous real-valued functions. As a kind of new proposal, we propose a method of representation of the oriented $1$-dimensional cell complexes by a $1$-dimensional CW complex whose $1$-cells ({\it edges}) are oriented.

In the third section, we discuss examples for our proposal and this is our main result (Theorems \ref{thm:1} and \ref{thm:2}). For example, we respect \cite{kitazawa9}. There cases presented above are first studied with the constraint that both functions $c_1$ and $c_2$ converge to a same number (Definition \ref{def:5}). Last, we also discuss canonical compactifications of the functions in Theorems \ref{thm:1} and \ref{thm:2} shortly (Theorems \ref{thm:3} and \ref{thm:4}). Further precise studies on these compactications are left to us as our future studies. 

Again, some knowledge and arguments of the presented preprints may help us to understand main ingredients of our arguments, where we do not assume them.

\section{Cell complexes and $0$- or $1$-dimensional cases with natural orientations induced by continuous real-valued functions.}
\subsection{Cell complexes and $0$- or $1$-dimensional cases.}
Hereafter, we use $D^0:={\mathbb{R}}^0$ for the one-point set and $S^{-1}$ for the empty set $\emptyset$.
A {\it cell complex} is a pair of a Hausdorff space $X$ and a family $\{e_{j,{\lambda}_j}:D^j-S^{j-1} \rightarrow X\}$ of homeomorphisms onto subsets of $X$ whose images are mutually disjoint and the union of all of whose images is $X$, and which satisfying the condition that the set ${\overline{e_{j,{\lambda}_j}}}^X-e_{j,{\lambda}_j}$ is always the union of the images of some of maps $e_{j^{\prime},{\lambda}_{j^{\prime}}}$ with $j^{\prime}<j$. Each map is a {\it cell} of the cell complex and the image of the map is also a {\it cell} of this. As a specific class, a {\it CW complex} is a cell complex which satisfies the following. 
\begin{itemize}
\item For each cell of its, its closure in $X$ contains finitely many cells.  
\item Each subset $O \subset X$ is open in $X$ if and only if for each cell $e_{j,{\lambda}_j}(D^j-S^{j-1})$ of it, the intersection $e_{j,{\lambda}_j}(D^j-S^{j-1}) \bigcap O$ is open in $e_{j,{\lambda}_j}(D^j-S^{j-1})$.
\end{itemize}
It is well-known that the dimensions of CW complexes yield a topological invariant for topological spaces regarded as CW complexes. In our paper, a cell complex is of dimension $0$ or $1$ and its dimension is a topological invariant and we call this a {\it weakly almost graph with ends and loops}. In this case, a $1$-cell is called an {\it edge} of the cell complex and a $0$-cell is called a {\it vertex} of this.

A weakly almost graph with ends and loops the closures of some of whose edges are not homeomorphic to $S^1$ is a {\it weakly almost graph with ends}. It is a {\it weakly almost graph} if the closures of some of its edges are always homeomorphic to $D^1$. For a weakly almost graph with ends and loops, we use the terminology "{\it graph}" instead of "weakly almost graph" if the $0$- or $1$-dimensional CW complex is locally finite.

For a weakly almost graph with ends, consider a continuous real-valued function which is injective on each of its edges and we call the pair a {\it digraph}, here.
\subsection{Pre-digraphs.}
A {\it pre-digraph} is defined first in the present paper. This is defined in the following way. 
\begin{Def}
\label{def:1}
A {\it pre-digraph} is a triplet $(G_0,S_{G_0},c_{G_0})$
 of a Hausdorff space $G_0$, its subset $S_{G_0}$ and a continuous function $c_{G_0}:G_0 \rightarrow \mathbb{R}$ satisfying the following.
\begin{itemize}
\item $G_0-S_{G_0}$ is homeomorphic to a disjoint union ${\sqcup}_{e \in E_{G_0,S_{G_0}}} e$ of copies of $D^1-S^0$. Note that each of the copies is homeomorphic to $\mathbb{R}$. By decomposing $G_0$ into the disjoint union of these copies and all elements of $S_{G_0}$, $G_0$ is regarded as a $0$- or $1$-dimensional cell complex with the set $S_{G_0}$ being the set of all $0$-cells of it.
\item On each copy before, the function $c_{G_0}$ is injective.
\end{itemize}
\end{Def}
\begin{Def}
\label{def:2}
Let $(G_0,S_{G_0},c_{G_0})$ be a pre-digraph.
A point $s \in G_0$ is {\it non-finite} or {\it NF} if we cannot choose a neighborhood $N_{s,G_0,s_{G_0}}$ of $G_0$ which is homeomorphic to a graph and has no point of $S_{G_0}-\{s\}$. Let ${\rm NF}(G_0,S_{G_0})$ denote the set of all NF points there.
\end{Def}

\begin{Def}
\label{def:3}
Two pre-digraphs $(G_{0,1},S_{G_{0,1}},c_{G_{0,1}})$ and $(G_{0,2},S_{G_{0,2}},c_{G_{0,2}})$ are {\it isomorphic} or equivalently, a pre-digraph is {\it isomorphic} to the other, if there exists an homeomorphism ${\Phi}_G:G_{0,1} \rightarrow G_{0,2}$ which maps $S_{G_{0,1}}$ onto $S_{G_{0,2}}$ and which yields the relation ${\phi}_G \circ c_{G_{0,1}}=c_{G_{0,2}} \circ {\Phi}_G$ with some homeomorphism ${\phi}_G:\mathbb{R} \rightarrow \mathbb{R}$ preserving the orientations and the orders for $\mathbb{R}$. Such a pair of homeomorphisms is said to be an {\it isomorphism} of the pre-digraphs. 
\end{Def}
\subsection{Graphs having oriented edges defined canonically from pre-digraphs.}
For a pre-digraph $(G_0,S_{G_0},c_{G_0})$, $(G_0-{\rm NF}(G_0,S_{G_0}),S_{G_0} \bigcap (G_0-{\rm NF}(G_0,S_{G_0})),c_{G_0} {\mid}_{G_0-{\rm NF}(G_0,S_{G_0})})$ is also a pre-digraph. By our definitions, $G_0-{\rm NF}(G_0,S_{G_0})$ is naturally a graph with ends, with $S_{G_0} \bigcap (G_0-{\rm NF}(G_0,S_{G_0}))$ being the set of all of its vertices. On the set of all connected components of $G_0-{\rm NF}(G_0,S_{G_0})$, we can define an equivalence relation ${\sim}_{G_0-{\rm NF}(G_0,S_{G_0}),c_{G_0}}$ generated by the rule as follows.
For two connected components $C_{G_0-{\rm NF}(G_0,S_{G_0}),c_{G_0},1}$ and $C_{G_0-{\rm NF}(G_0,S_{G_0}),c_{G_0},2}$ of $G_0-{\rm NF}(G_0,S_{G_0})$, if and only if at least one of the following holds, then $C_{G_0-{\rm NF}(G_0,S_{G_0}),c_{G_0},1} {\sim}_{G_0-{\rm NF}(G_0,S_{G_0}),c_{G_0}} C_{G_0-{\rm NF}(G_0,S_{G_0}),c_{G_0},2}$. Hereafter, we abuse the notation of Definition \ref{def:1}.
\begin{itemize}
\item There exist non-empty families $\{e_{1,j_1}\}_{j_1 \in J_1}, \{e_{2,j_2}\}_{j_2 \in J_2} \subset E_{G_0,S_{G_0}}$
 such that ${\sqcup}_{j_1 \in J_1} e_{1,j_1} \subset C_{G_0-{\rm NF}(G_0,S_{G_0}),c_{G_0},1}$ and ${\sqcup}_{j_2 \in J_2} e_{2,j_2} \subset C_{G_0-{\rm NF}(G_0,S_{G_0}),c_{G_0},2}$, that the closures of ${\sqcup}_{j_1 \in J_1} e_{1,j_1}$ and ${\sqcup}_{j_2 \in J_2} e_{2,j_2}$ in $G_0$ have a same point $s_{e_1,e_2} \in {\rm NF}(G_0,S_{G_0})$ in common, and that the values of $c_{G_0} {\mid}_{\overline{{\sqcup}_{j_1 \in J_1} e_{1,j_1}}^{G_0}}$ and $c_{G_0} {\mid}_{\overline{{\sqcup}_{j_1 \in J_1} e_{2,j_2}}^{G_0}}$ are maximal at $s_{e_1,e_2}$. 
\item There exist non-empty families $\{e_{1,j_1}\}_{j_1 \in J_1}, \{e_{2,j_2}\}_{j_2 \in J_2} \subset E_{G_0,S_{G_0}}$,  such that ${\sqcup}_{j_1 \in J_1} e_{1,j_1} \subset C_{G_0-{\rm NF}(G_0,S_{G_0}),c_{G_0},1}$ and ${\sqcup}_{j_2 \in J_2} e_{2,j_2} \subset C_{G_0-{\rm NF}(G_0,S_{G_0}),c_{G_0},2}$, that the closures of ${\sqcup}_{j_1 \in J_1} e_{1,j_1}$ and ${\sqcup}_{j_2 \in J_2} e_{2,j_2}$ in $G_0$ have a same point $s_{e_1,e_2} \in {\rm NF}(G_0,S_{G_0})$ in common, and that the values of $c_{G_0} {\mid}_{\overline{{\sqcup}_{j_1 \in J_1} e_{1,j_1}}^{G_0}}$ and $c_{G_0} {\mid}_{\overline{{\sqcup}_{j_1 \in J_1} e_{2,j_2}}^{G_0}}$ are minimal at $s_{e_1,e_2}$. 
\end{itemize}
Note that only this rule we may not have an equivalence relation. We consider the unique equivalence relation generated by this. 
We can define a $0$-dimensional or $1$-dimensional cell complex whose edges are oriented.
\begin{itemize}
\item The set of all vertices of it is the set of all equivalence classes $[C_{G_0-{\rm NF}(G_0,S_{G_0}),c_{G_0}}]_{{\sim}_{G_0-{\rm NF}(G_0,S_{G_0}),c_{G_0}}}$.
\item The set of all edges of it is ${\rm NF}(G_0,S_{G_0})$ and each element of this set is oriented and incident to some vertex which is identified with some $[C_{G_0-{\rm NF}(G_0,S_{G_0}),c_{G_0}}]_{{\sim}_{G_0-{\rm NF}(G_0,S_{G_0}),c_{G_0}}}$ according to the rule as follows.
\begin{itemize}
\item An edge of the $1$-dimensional cell complex is an oriented edge departing from the vertex identified with $[C_{G_0-{\rm NF}(G_0,S_{G_0}),c_{G_0}}]_{{\sim}_{G_0-{\rm NF}(G_0,S_{G_0}),c_{G_0}}}$ if and only if there exist $C_{G_0-{\rm NF}(G_0,S_{G_0}),c_{G_0}} \in [C_{G_0-{\rm NF}(G_0,S_{G_0}),c_{G_0}}]_{{\sim}_{G_0-{\rm NF}(G_0,S_{G_0}),c_{G_0}}}$ and a non-empty family $\{e_j\}_{j \in J} \subset E_{G_0,S_{G_0}}$ satisfying $\{e_j\}_{j \in J} \subset C_{G_0-{\rm NF}(G_0,S_{G_0}),c_{G_0}}$ such that the value of $c_{G_0} {\mid}_{\overline{{\sqcup}_{j \in J} e_j}^{G_0}}$ is maximal at some point of $C_{G_0-{\rm NF}(G_0,S_{G_0}),c_{G_0}}$. 
\item An edge of the $1$-dimensional cell complex is an oriented edge entering the vertex identified with $[C_{G_0-{\rm NF}(G_0,S_{G_0}),c_{G_0}}]_{{\sim}_{G_0-{\rm NF}(G_0,S_{G_0}),c_{G_0}}}$ if and only if there exist $C_{G_0-{\rm NF}(G_0,S_{G_0}),c_{G_0}} \in [C_{G_0-{\rm NF}(G_0,S_{G_0}),c_{G_0}}]_{{\sim}_{G_0-{\rm NF}(G_0,S_{G_0}),c_{G_0}}}$ and a non-empty family $\{e_j\}_{j \in J} \subset E_{G_0,S_{G_0}}$ satisfying $\{e_j\}_{j \in J} \subset C_{G_0-{\rm NF}(G_0,S_{G_0}),c_{G_0}}$  such that the value of $c_{G_0} {\mid}_{\overline{{\sqcup}_{j \in J} e_j}^{G_0}}$ is minimal at some point of $C_{G_0-{\rm NF}(G_0,S_{G_0}),c_{G_0}}$. 
\end{itemize}
\end{itemize} 
\begin{Def}
\label{def:4}
We call this the {\it graph diagram for the NF case} or {\it GDNF} of a pre-digraph $(G_0,S_{G_0},c_{G_0})$. We use ${\rm GDNF}(G_0,S_{G_0},c_{G_0})$ for the GDNF of $(G_0,S_{G_0},c_{G_0})$.
\end{Def}

\section{Main Theorems.}
We abuse the notation of the end of the first section, or our introduction. We define a certain class of pairs $(c_1,c_2)$ of the real analytic functions and the $c_1<c_2$-zero-set $(c_1,c_2,D_{c_1,c_2},X_{m,c_1,c_2})$, as presented there. 

\begin{Def}
\label{def:5}
The $c_1<c_2$-zero-set of the real analytic functions is of {\it same asymptotic behaviors at infinities} with respect to the topology category or {\it SAB-at-Infinity-in-Top} if the following hold.
\begin{enumerate}
\item At the infinity $-\infty$, either of the following hold.
\begin{enumerate}
\item The values $c_1(x)$ and $c_2(x)$ converge to a same real number as $x$ diverges to $-\infty$.
\item The values $c_1(x)$ and $c_2(x)$ diverge to $-\infty$ 
as $x$ diverges to $-\infty$.
\item The values $c_1(x)$ and $c_2(x)$ diverge to $+\infty$ 
as $x$ diverges to $-\infty$.
\end{enumerate}
\item At the infinity $+\infty$, either of the following hold.
\begin{enumerate}
\item The values $c_1(x)$ and $c_2(x)$ converge to a same real number as $x$ diverges to $+\infty$.
\item The values $c_1(x)$ and $c_2(x)$ diverge to $-\infty$ 
as $x$ diverges to $+\infty$.
\item The values $c_1(x)$ and $c_2(x)$ diverge to $+\infty$ 
as $x$ diverges to $+\infty$.
\end{enumerate}
\end{enumerate}
\end{Def} 

\subsection{Patterns of ${\rm GDNF}(R_{{\pi}_{m+1,1} {\mid}_{X_{m,c_1,c_2}}},q_{{\pi}_{m+1,1} {\mid}_{X_{m,c_1,c_2}}}(S({\pi}_{m+1,1} {\mid}_{X_{m,c_1,c_2}})),\bar{{\pi}_{m+1,1} {\mid}_{X_{m,c_1,c_2}}})$ for $(c_1,c_2,D_{c_1,c_2},X_{m,c_1,c_2})$ of SAB-at-Infinity-in-Top.}
Theorem \ref{thm:1} (Theorem \ref{thm:1} (\ref{thm:1.2})) is our new result.
\begin{Thm}
\label{thm:1}
For $(c_1,c_2,D_{c_1,c_2},X_{m,c_1,c_2})$ of SAB-at-Infinity-in-Top, the following hold.
\begin{enumerate}
\item \label{thm:1.1} $(R_{{\pi}_{m+1,1} {\mid}_{X_{m,c_1,c_2}}},q_{{\pi}_{m+1,1} {\mid}_{X_{m,c_1,c_2}}}(S({\pi}_{m+1,1} {\mid}_{X_{m,c_1,c_2}})),\bar{{\pi}_{m+1,1} {\mid}_{X_{m,c_1,c_2}}})$ is a pre-digraph.
\item \label{thm:1.2} ${\rm GDNF}(R_{{\pi}_{m+1,1} {\mid}_{X_{m,c_1,c_2}}},q_{{\pi}_{m+1,1} {\mid}_{X_{m,c_1,c_2}}}(S({\pi}_{m+1,1} {\mid}_{X_{m,c_1,c_2}})),\bar{{\pi}_{m+1,1} {\mid}_{X_{m,c_1,c_2}}})$ contains no edge whose closure is homeomorphic to $S^1$.
\item \label{thm:1.3}
${\rm GDNF}(R_{{\pi}_{m+1,1} {\mid}_{X_{m,c_1,c_2}}},q_{{\pi}_{m+1,1} {\mid}_{X_{m,c_1,c_2}}}(S({\pi}_{m+1,1} {\mid}_{X_{m,c_1,c_2}})),\bar{{\pi}_{m+1,1} {\mid}_{X_{m,c_1,c_2}}})$ 
 is isomorphic to either of the following. Furthermore, we can have each case by choosing elementary functions $c_1$ and $c_2$ suitably.
\begin{enumerate}
\item \label{thm:1.3.1} A one-point set.
\item \label{thm:1.3.2} A $1$-dimensional connected cell complex with exactly one vertex and exactly one edge whose closure taken in the cell complex is homeomorphic to $\{t>0\}$.
\item \label{thm:1.3.3} A $1$-dimensional connected cell complex with exactly two vertices and exactly one edge whose closure taken in the cell complex is homeomorphic to $D^1$.

\item \label{thm:1.3.4} A $1$-dimensional connected cell complex with exactly three vertices and exactly two edges whose closures taken in the cell complex are homeomorphic to $D^1$ and which are both oriented as edges departing from one vertex.
\item \label{thm:1.3.5} A $1$-dimensional connected cell complex with exactly three vertices and exactly two edges whose closures taken in the cell complex are homeomorphic to $D^1$ and which are both oriented as edges entering one vertex.
\item \label{thm:1.3.6} A $1$-dimensional connected cell complex with exactly three vertices and exactly two edges whose closures taken in the cell complex are homeomorphic to $D^1$ and which are oriented as an edge departing from one vertex and an edge entering the vertex, respectively.
\end{enumerate}
	\end{enumerate}
\end{Thm}
\begin{proof}
We prove the theorem in several steps. \\

\noindent STEP 1-1 Proving (\ref{thm:1.1}). \\
\ \\
We prove (\ref{thm:1.1}).
We choose two distinct elements $v_1$ and $v_2$ of $R_{{\pi}_{m+1,1} {\mid}_{X_{m,c_1,c_2}}}$ and
 consider the preimages ${q_{{\pi}_{m+1,1} {\mid}_{X_{m,c_1,c_2}}}}^{-1}(v_i)$ and consider
 small connected neighborhoods of ${\pi}_{m+1,2}({q_{{\pi}_{m+1,1} {\mid}_{X_{m,c_1,c_2}}}}^{-1}(v_i))$ in ${\overline{D_{c_1,c_2}}}^{{\mathbb{R}}^2}$ represented as single connected components $C_{v_i,D_{c_1,c_2}}$ of the intersections of  ${\overline{D_{c_1,c_2}}}^{{\mathbb{R}}^2}$ and small closed sets in ${\mathbb{R}}^2$ of the form $\{t_{v_i,1} \leq t \leq t_{v_i,2}\} \times \mathbb{R}$. Due to our situation, we can choose these neighborhoods disjointly. This means that the Reeb space $R_{{\pi}_{m+1,1} {\mid}_{X_{m,c_1,c_2}}}$ is a Hausdorff space.

From our situation, we can see that except at most two elements of $R_{{\pi}_{m+1,1} {\mid}_{X_{m,c_1,c_2}}}$, for each element $v$ there, ${\pi}_{m+1,2}({q_{{\pi}_{m+1,1} {\mid}_{X_{m,c_1,c_2}}}}^{-1}(v))$ is represented as a closed interval of the form $\{t_v\} \times \{t_{v,1} \leq t \leq t_{v,2}\}$ with three suitably chosen real numbers $t_{v_1} \leq t_{v_2}$ and $t_v$ some of which may agree. By removing these two exceptional elements ${v_1}^{\prime}$ and ${v_2}^{\prime}$, we have a space homeomorphic to a graph whose vertex is defined by the rule for Reeb graphs. This is due to the fact that the restriction of ${\pi}_{m+1,1} {\mid}_{X_{m,c_1,c_2}}$ to $X_{m,c_1,c_2}-{({\pi}_{m+1,1} {\mid}_{X_{m,c_1,c_2}})}^{-1}(\{{v_1}^{\prime},{v_2}^{\prime}\})$ is proper and main ingredients including important result of \cite{saeki1, saeki2}. More precisely,\\
${\pi}_{m+1,1} {\mid}_{X_{m,c_1,c_2}-{({\pi}_{m+1,1} {\mid}_{X_{m,c_1,c_2}})}^{-1}(\{{v_1}^{\prime},{v_2}^{\prime}\})}(S({\pi}_{m+1,1} {\mid}_{X_{m,c_1,c_2}-{({\pi}_{m+1,1} {\mid}_{X_{m,c_1,c_2}})}^{-1}(\{{v_1}^{\prime},{v_2}^{\prime}\})}) \bigcap C_{v_i,D_{c_1,c_2}})$ is a closed and discrete (finite) subset in $\mathbb{R}$, for each set $C_{v_i,D_{c_1,c_2}}$ defined in the way above, and we apply related important result and arguments of \cite{saeki1,saeki2}. 

We consider at most two elements ${v_j}^{\prime}$  of $R_{{\pi}_{m+1,1} {\mid}_{X_{m,c_1,c_2}}}$, presented above. For such an element, ${\pi}_{m+1,2}({q_{{\pi}_{m+1,1} {\mid}_{X_{m,c_1,c_2}}}}^{-1}({v_j}^{\prime}))$ is represented as either of the form $\{t_{{v_j}^{\prime}}\} \times \{t \geq t_{{v_j}^{\prime},0}\}$,  $\{t_{{v_j}^{\prime}}\} \times \{t \leq t_{{v_j}^{\prime},0}\}$, or $\{t_{{v_j}^{\prime}}\} \times \mathbb{R}$, for a suitably chosen real number $t_{{v_j}^{\prime}}$, and this set is a $1$-dimensional smooth connected and non-compact submanifold and a closed subset of ${\mathbb{R}}^2$. Let such a connected and non-compact closed subset be denoted by $T_{{v_j}^{\prime}} \subset {\mathbb{R}}^2$. We consider a small connected neighborhood of ${\pi}_{m+1,2}({q_{{\pi}_{m+1,1} {\mid}_{X_{m,c_1,c_2}}}}^{-1}({v_j}^{\prime}))$ in ${\overline{D_{c_1,c_2}}}^{{\mathbb{R}}^2}$ represented as a single connected component $C_{{v_j}^{\prime},D_{c_1,c_2}}$, defined in the argument above, and containing such a small
closed set $T_{{v_j}^{\prime}} \subset {\mathbb{R}}^2$. We can have a case such that $C_{{v_j}^{\prime},D_{c_1,c_2}}$ contains exactly one connected component of ${{\pi}_{2,1}}^{-1}(t_{{v_j}^{\prime}}) \bigcap {\overline{D_{c_1,c_2}}}^{{\mathbb{R}}^2}$. From our situation, $q_{{\pi}_{m+1,1} {\mid}_{X_{m,c_1,c_2}}}({{\pi}_{m+1,2}}^{-1}(C_{{v_j}^{\prime},D_{c_1,c_2}}))$ is a closed and connected neighborhood of ${v_j}^{\prime}$ in $R_{{\pi}_{m+1,1} {\mid}_{X_{m,c_1,c_2}}}$ and the closure of $q_{{\pi}_{m+1,1} {\mid}_{X_{m,c_1,c_2}}}({{\pi}_{m+1,2}}^{-1}(C_{{v_j}^{\prime},D_{c_1,c_2}}))-\{{v_j}^{\prime}\}$ in $q_{{\pi}_{m+1,1} {\mid}_{X_{m,c_1,c_2}}}({{\pi}_{m+1,2}}^{-1}(C_{{v_j}^{\prime},D_{c_1,c_2}}))$ is $q_{{\pi}_{m+1,1} {\mid}_{X_{m,c_1,c_2}}}({{\pi}_{m+1,2}}^{-1}(C_{{v_j}^{\prime},D_{c_1,c_2}}))$.


These arguments complete the proof of (\ref{thm:1.1}).

Note that we have also presented most of the arguments for proving (\ref{thm:1.1}) in the previous preprints of the author, especially in \cite{kitazawa11, kitazawa12} (\cite{kitazawa9}). We have argued in a self-contained way, here. \\
\ \\
\noindent STEP 1-2 Exposition on (\ref{thm:1.2}). \\
\ \\
 $T_{{v_j}^{\prime}} \subset {\mathbb{R}}^2$ is represented as either of the form $\{t_{{v_j}^{\prime}}\} \times \{t \geq t_{{v_j}^{\prime},0}\}$,  $\{t_{{v_j}^{\prime}}\} \times \{t \leq t_{{v_j}^{\prime},0}\}$, or $\{t_{{v_j}^{\prime}}\} \times \mathbb{R}$, for a suitably chosen real number $t_{{v_j}^{\prime}}$, and this set is a $1$-dimensional smooth connected and non-compact submanifold and a closed subset of ${\mathbb{R}}^2$. 
By our definition, ${\rm GDNF}(R_{{\pi}_{m+1,1} {\mid}_{X_{m,c_1,c_2}}},q_{{\pi}_{m+1,1} {\mid}_{X_{m,c_1,c_2}}}(S({\pi}_{m+1,1} {\mid}_{X_{m,c_1,c_2}})),\bar{{\pi}_{m+1,1} {\mid}_{X_{m,c_1,c_2}}})$ contains no edge whose closure is homeomorphic to $S^1$. \\
\ \\

\noindent STEP 1-3 Proving (\ref{thm:1.3}). \\
\ \\
We prove (\ref{thm:1.3}). This is a new ingredient of the present paper. We also apply arguments presented in the previous preprint of the author, in a self-contained way. 

The (at most two) elements ${v_j}^{\prime}$  of $R_{{\pi}_{m+1,1} {\mid}_{X_{m,c_1,c_2}}}$ can represent at most two edges of ${\rm GDNF}(R_{{\pi}_{m+1,1} {\mid}_{X_{m,c_1,c_2}}},q_{{\pi}_{m+1,1} {\mid}_{X_{m,c_1,c_2}}}(S({\pi}_{m+1,1} {\mid}_{X_{m,c_1,c_2}})),\bar{{\pi}_{m+1,1} {\mid}_{X_{m,c_1,c_2}}})$ and ${\rm GDNF}(R_{{\pi}_{m+1,1} {\mid}_{X_{m,c_1,c_2}}},q_{{\pi}_{m+1,1} {\mid}_{X_{m,c_1,c_2}}}(S({\pi}_{m+1,1} {\mid}_{X_{m,c_1,c_2}})),\bar{{\pi}_{m+1,1} {\mid}_{X_{m,c_1,c_2}}})$ has at most two edges. \\
\ \\
STEP 1-3-1 The case of the non-existence of edges of\\ ${\rm GDNF}(R_{{\pi}_{m+1,1} {\mid}_{X_{m,c_1,c_2}}},q_{{\pi}_{m+1,1} {\mid}_{X_{m,c_1,c_2}}}(S({\pi}_{m+1,1} {\mid}_{X_{m,c_1,c_2}})),\bar{{\pi}_{m+1,1} {\mid}_{X_{m,c_1,c_2}}})$. \\
\ \\
This is for (\ref{thm:1.3.1}). \\
This is realized by a constant function $c_1(x):=0$ and $c_2(x)=\frac{1}{x^2+1}$. This is for the case (\ref{thm:1.3.1}). \\
\ \\
STEP 1-3-2 The case of the existence of exactly one edge of\\ ${\rm GDNF}(R_{{\pi}_{m+1,1} {\mid}_{X_{m,c_1,c_2}}},q_{{\pi}_{m+1,1} {\mid}_{X_{m,c_1,c_2}}}(S({\pi}_{m+1,1} {\mid}_{X_{m,c_1,c_2}})),\bar{{\pi}_{m+1,1} {\mid}_{X_{m,c_1,c_2}}})$.\\
\ \\
First, in the case where exactly one edge of\\ ${\rm GDNF}(R_{{\pi}_{m+1,1} {\mid}_{X_{m,c_1,c_2}}},q_{{\pi}_{m+1,1} {\mid}_{X_{m,c_1,c_2}}}(S({\pi}_{m+1,1} {\mid}_{X_{m,c_1,c_2}})),\bar{{\pi}_{m+1,1} {\mid}_{X_{m,c_1,c_2}}})$ exists, the number of vertices of\\ ${\rm GDNF}(R_{{\pi}_{m+1,1} {\mid}_{X_{m,c_1,c_2}}},q_{{\pi}_{m+1,1} {\mid}_{X_{m,c_1,c_2}}}(S({\pi}_{m+1,1} {\mid}_{X_{m,c_1,c_2}})),\bar{{\pi}_{m+1,1} {\mid}_{X_{m,c_1,c_2}}})$ can be easily shown to be $1$ or $2$. Second, in this case, for the existence of exactly one vertex of\\ ${\rm GDNF}(R_{{\pi}_{m+1,1} {\mid}_{X_{m,c_1,c_2}}},q_{{\pi}_{m+1,1} {\mid}_{X_{m,c_1,c_2}}}(S({\pi}_{m+1,1} {\mid}_{X_{m,c_1,c_2}})),\bar{{\pi}_{m+1,1} {\mid}_{X_{m,c_1,c_2}}})$, either $c_1$ or $c_2$ must be constant and this is for STEP 2-2-1 and (\ref{thm:1.3.2}). \\
\ \\
STEP 1-3-2-1 For (\ref{thm:1.3.2}). \\
\ \\
Suppose that there exists exactly one vertex of\\ ${\rm GDNF}(R_{{\pi}_{m+1,1} {\mid}_{X_{m,c_1,c_2}}},q_{{\pi}_{m+1,1} {\mid}_{X_{m,c_1,c_2}}}(S({\pi}_{m+1,1} {\mid}_{X_{m,c_1,c_2}})),\bar{{\pi}_{m+1,1} {\mid}_{X_{m,c_1,c_2}}})$.

For this case, we define the constant function $c_1(x):=0$ and another real analytic function $c_2(x):=\frac{1}{x^2+1}-\frac{\sin (e^{x^2})}{{(x^2+1)}^2}>0$. 
We consider the 1st derivative ${c_2}^{\prime}:\mathbb{R} \rightarrow \mathbb{R}$ for the function $c_2$. We have ${c_2}^{\prime}(x)=-\frac{2x}{{(x^2+1)}^2}-\frac{2xe^{x^2}(\cos ({e^{x^2}})){(x^2+1)}^2-4x\sin (e^{x^2}) \times (x^2+1)}{{(x^2+1)}^4}$. From this and the orders on the divergence for functions at the infinities $-\infty$ and $+\infty$, $c_2(x)>0$ for some positive number $R>0$ and any real number $x$ satisfying $|x|>R$ and the sets $\{x \mid {c_2}^{\prime}(x)>0\}$, $\{x \mid {c_2}^{\prime}(x)<0\}$ and $\{x \mid {c_2}^{\prime}(x)=0\}$ are all unbounded. From these arguments, these functions $c_1$ and $c_2$ are for the case (\ref{thm:1.3.2}). Related arguments have appeared first in \cite{kitazawa11} and have appeared in \cite{kitazawa12}. Here, we explain them in a self-contained way and do not assume related knowledge. \\
\ \\
STEP 1-3-2-2 For (\ref{thm:1.3.3}). \\
\ \\
Suppose that there exist exactly two vertices of\\ ${\rm GDNF}(R_{{\pi}_{m+1,1} {\mid}_{X_{m,c_1,c_2}}},q_{{\pi}_{m+1,1} {\mid}_{X_{m,c_1,c_2}}}(S({\pi}_{m+1,1} {\mid}_{X_{m,c_1,c_2}})),\bar{{\pi}_{m+1,1} {\mid}_{X_{m,c_1,c_2}}})$. We can see
that this case is realized by functions with $c_1(x):=-\frac{1}{x^2+1}<0$ and $c_2(x):=\frac{1}{x^2+1}-\frac{\sin (e^{x^2})}{{(x^2+1)}^2}>0$. Remember the derivative ${c_2}^{\prime}$ and behavior of related functions, again. \\
\ \\
STEP 1-3-3 The case of the existence of exactly two edges of\\ ${\rm GDNF}(R_{{\pi}_{m+1,1} {\mid}_{X_{m,c_1,c_2}}},q_{{\pi}_{m+1,1} {\mid}_{X_{m,c_1,c_2}}}(S({\pi}_{m+1,1} {\mid}_{X_{m,c_1,c_2}})),\bar{{\pi}_{m+1,1} {\mid}_{X_{m,c_1,c_2}}})$.\\
\ \\

Here, we also need arguments on the derivatives ${c_j}^{\prime}:\mathbb{R} \rightarrow \mathbb{R}$ for the functions $c_1$ and $c_2$, which are also originally from \cite{kitazawa11, kitazawa12} again. For this case, we also present new explicit arguments. Of course we discuss these functions in a self-contained way. \\
\ \\
STEP 1-3-3-1 For (\ref{thm:1.3.4}, \ref{thm:1.3.5}). \\
\ \\

We consider $c_1(x):=-\frac{1}{x^2+1}+\frac{x^2\sin (e^{x^2})}{{(x^2+1)}^k}$ with a sufficiently large positive integer $k$ and $c_2(x):=\frac{1}{x^2+1}-\frac{3}{2{(x^2+1)}^2}$. 
Both the values $c_1(x)$ and $c_2(x)$ converge to $0$ as $x$ diverges to $\pm \infty$. We also have $c_1(x)<c_2(x)$ for any $x \in \mathbb{R}$. We also have $c_2(0)=-\frac{1}{2}<0$. We also have 
$c_1(x)<0$ and $c_2(x)>0$ for some real number $R>0$ and any real number $x$ satisfying $|x|>R$.
We can understand the form of the derivative ${c_1}^{\prime}$ as in STEP 2-2-1 and have similar properties: the sets $\{x \mid {c_1}^{\prime}(x)>0\}$, $\{x \mid {c_1}^{\prime}(x)<0\}$ and $\{x \mid {c_1}^{\prime}(x)=0\}$ are all unbounded. 
By investigating the global behavior, we can see that this is for (\ref{thm:1.3.4}). By exchanging the signs, we have an example for (\ref{thm:1.3.5}). \\
\ \\
STEP 1-3-3-2 For (\ref{thm:1.3.6}). \\
\ \\ In this case, we respect the real-valued function $c_{0,0}(x):=\frac{1}{x^2+1}$. 

We consider $c_1(x):=-\frac{1}{{(e^{-x}-1)}^2+1}$ and $c_{2,0}(x):=\frac{1}{{(e^{x}-1)}^2+1}-\frac{1}{2}$.

We can see $c_1(x)<c_{2,0}(x)$ for $x \in \mathbb{R}$. We can show this by considering the case $x \leq 0$ and the case $x \geq 0$. Both the values $c_1(x)$ and $c_{2,0}(x)$ converge to $0$ as $x$ diverges to $-\infty$. Both the values $c_1(x)$ and $c_{2,0}(x)$ converge to $-\frac{1}{2}$ as $x$ diverges to $\infty$. 
Furthermore, there exists a positive number $R>0$ and for any number $x$ satisfying $x<R$, $c_1(x)<0<c_{2,0}(x)$ and for any number satisfying $x>R$, $c_1(x)<-\frac{1}{2}<c_{2,0}(x)$.
 
We define $t_1:=e^{-x}-1$ and $t_2:=e^{x}-1$ and $c_2(x):=c_{2,0}(x)+\frac{{t_1}^2{\sin}^2 (e^{{t_1}^2})}{{({t_1}^2+1)}^k}+\frac{{t_2}^2{\sin}^2 (e^{{t_2}^2})}{{({t_2}^2+1)}^k}$ for some sufficiently large positive integer $k$. 
Here we remember the relation ${\sin}^2 a=\frac{1-\cos {2a}}{2}$. We also remember the real-valued function $c _{0}(x):=\frac{x^2\sin (e^{x^2})}{{(x^2+1)}^k}$, which is defined in STEP 2-3-1. By applying similar arguments, we have similar properties of the function.

By investigating the global behavior, we can see that this pair $(c_1,c_2)$ is for  (\ref{thm:1.3.6}).  \\
\ \\
STEP 1-3-3-3 The non-existence of cases except (\ref{thm:1.3.4}, \ref{thm:1.3.5}, \ref{thm:1.3.6}) . \\

We must show the non-existence of cases except (\ref{thm:1.3.4}, \ref{thm:1.3.5}, \ref{thm:1.3.6}) to complete the proof .

Suppose that the closure of an edge of ${\rm GDNF}(R_{{\pi}_{m+1,1} {\mid}_{X_{m,c_1,c_2}}},q_{{\pi}_{m+1,1} {\mid}_{X_{m,c_1,c_2}}}(S({\pi}_{m+1,1} {\mid}_{X_{m,c_1,c_2}})),\bar{{\pi}_{m+1,1} {\mid}_{X_{m,c_1,c_2}}})$ is not homeomorphic to $D^1$. By STEP 2 or (\ref{thm:1.2}) and our definition, it must be homeomorphic to $\{t \geq 0\}$ and as an argument for STEP 3-2-1, either $c_1$ and $c_2$ must be constant. This is for STEP 3-2-1 or (\ref{thm:1.3.2}). This is a contradiction. 

Suppose that these two edges of ${\rm GDNF}(R_{{\pi}_{m+1,1} {\mid}_{X_{m,c_1,c_2}}},q_{{\pi}_{m+1,1} {\mid}_{X_{m,c_1,c_2}}}(S({\pi}_{m+1,1} {\mid}_{X_{m,c_1,c_2}})),\bar{{\pi}_{m+1,1} {\mid}_{X_{m,c_1,c_2}}})$ are oriented as edges departing from a same vertex or entering another same vertex. Two sets $T_{{v_j}^{\prime}} \subset {\mathbb{R}}^2$ are represented by the form $\{t_{{v_j}^{\prime},1}\} \times \{t \geq t_{{v_j}^{\prime},1,0}\}$,  $\{t_{{v_j}^{\prime},2}\} \times \{t \leq t_{{v_j}^{\prime},2,0}\}$ with $t_{{v_j}^{\prime},2,0}<t_{{v_j}^{\prime},1,0}$. This implies the case (\ref{thm:1.3.4}) or (\ref{thm:1.3.5}). This is also regarded as a contradiction.
\ \\
 \ \\
This completes the proof of (\ref{thm:1.3}) and the present theorem.
\end{proof}
\subsection{Natural specific classes of pre-digraphs and the $c_1<c_2$-zero-sets of the real analytic functions of SAB-at-Infinity-in-Top.}

By abusing the notation, we define a natural specific class of the $c_1<c_2$-zero-sets of the real analytic functions of SAB-at-Infinity-in-Top.
\begin{Def}
\label{def:6}
For a {\it pre-digraph} $(G_0,S_{G_0},c_{G_0})$, if to each element of the set of all equivalence classes $[C_{G_0-{\rm NF}(G_0,S_{G_0}),c_{G_0}}]_{{\sim}_{G_0-{\rm NF}(G_0,S_{G_0}),c_{G_0}}}$, exactly one connected component $C_{G_0-{\rm NF}(G_0,S_{G_0}),c_{G_0}}$ belongs, then the pre-digraph is {\it normal}.
\end{Def}
\begin{Def}
\label{def:7}
The $c_1<c_2$-zero-set of the real analytic functions is {\it normal} if\\
 $(R_{{\pi}_{m+1,1} {\mid}_{X_{m,c_1,c_2}}},q_{{\pi}_{m+1,1} {\mid}_{X_{m,c_1,c_2}}}(S({\pi}_{m+1,1} {\mid}_{X_{m,c_1,c_2}})),\bar{{\pi}_{m+1,1} {\mid}_{X_{m,c_1,c_2}}})$\\ is a pre-digraph as in Theorem \ref{thm:1} (\ref{thm:1.1}) and normal.
\end{Def}
\begin{Thm}
\label{thm:2}
In Theorem \ref{thm:1}, we have the following.
\begin{enumerate}
\item \label{thm:2.1} In the cases of Theorem \ref{thm:1} {\em (}\ref{thm:1.3.1}{\rm )} and {\rm (}\ref{thm:1.3.2}{\rm )}, the $c_1<c_2$-zero-sets of the real analytic functions must be normal.
\item \label{thm:2.2} In the cases of Theorem \ref{thm:1} {\em (}\ref{thm:1.3.3}{\rm )}, {\em (}\ref{thm:1.3.4}{\rm )} and {\rm (}\ref{thm:1.3.5}{\rm )}, the number of connected components $C_{G_0-{\rm NF}(G_0,S_{G_0}),c_{G_0}}$ belonging to each element of the set of all equivalence classes $[C_{G_0-{\rm NF}(G_0,S_{G_0}),c_{G_0}}]_{{\sim}_{G_0-{\rm NF}(G_0,S_{G_0}),c_{G_0}}}$ can be an arbitrary positive integer.
\end{enumerate}
\end{Thm}
\begin{proof}
We have the statement (\ref{thm:2.1}) by the definitions and construction in our proof of Theorem \ref{thm:1}: it is important that in the case of Theorem \ref{thm:1} (\ref{thm:1.3.2}), either $c_1$ or $c_2$ must be constant.

We prove the statement (\ref{thm:2.2}). For this, it is sufficient to revise our proof of Theorem \ref{thm:1} in a suitable way.

Hereafter, the product denoted by ${\prod}_{j=1}^{0} a_j$ is $1$, where $a_j$ is used for a suitable number.
 
We explain our revision of (\ref{thm:1.3.3}) of Theorem \ref{thm:1} and its proof (STEP 1-2-2-2). First we choose arbitrary positive integers $k_1$ and $k_2$, choose a sufficiently large number $R_{k_1,k_2}$ depending on them and sequences $\{a_{k_1,k_2,R_{k_1,k_2},i,j}\}_{j=1}^{k_i-1}$ of real numbers whose absolute values are greater than $R_{k_1,k_2}$ and mutually distinct for $\{a_{k_1,k_2,R_{k_1,k_2},1,j}\}_{j=1}^{k_1} \sqcup \{a_{k_1,k_2,R_{k_1,k_2},2,j}\}_{j=1}^{k_2}$. We use $c_1(x):=-\frac{{\prod}_{j=1}^{k_1-1} {(x-a_{k_1,k_2,R_{k_1,k_2},1,j})}^2}{{(x^2+1)}^{k_1}} \leq 0$ and $c_2(x):=\frac{{\prod}_{j=1}^{k_2-1} {(x-a_{k_1,k_2,R_{k_1,k_2},2,j})}^2}{{(x^2+1)}^{k_2}} \times (\frac{1}{x^2+1}-\frac{\sin (e^{x^2})}{{(x^2+1)}^2}) \geq 0$. We can see that $c_1(x)<c_2(x)$ for $x \in \mathbb{R}$.

We explain our revision of (\ref{thm:1.3.4}) and (\ref{thm:1.3.5}) of Theorem \ref{thm:1} and its proof (STEP 1-3-3-1).

First we choose arbitrary positive integers $k_1$ and $k_2$. We choose a sufficiently large number $R_{k_1,k_2}>0$ depending on them. We also choose a number $T_{k_1,k_2}>1$ sufficiently close to $1$, depending on $k_1$ and $k_2$. In addition, We also choose a sufficiently large integer $k_{k_1,k_2}>0$, depending on $k_1$ and $k_2$.

Depending on these numbers, we choose sequences $\{a_{k_1,k_2,R_{k_1,k_2},T_{k_1,k_2},k_{k_1,k_2},i,j}\}_{j=1}^{k_i-1}$ of real numbers whose absolute values are greater than $R_{k_1,k_2}$ and mutually distinct for $\{a_{k_1,k_2,R_{k_1,k_2},T_{k_1,k_2},k_{k_1,k_2}, 1,j}\}_{j=1}^{k_1} \sqcup \{a_{k_1,k_2,R_{k_1,k_2},T_{k_1,k_2},k_{k_1,k_2}, 2,j}\}_{j=1}^{k_2}$, and we use $c_1(x):=-\frac{{\prod}_{j=1}^{k_1-1} {(x-a_{k_1,k_2,R_{k_1,k_2},T_{k_1,k_2},k_{k_1,k_2}, 1,j})}^2}{{(x^2+1)}^{k_1}} \times (\frac{1}{x^2+1}-\frac{x^2\sin (e^{x^2})}{{(x^2+1)}^{k_{k_1,k_2}}}) \leq 0$ and $c_2(x):=\frac{{\prod}_{j=1}^{k_2-1} {(x-a_{k_1,k_2, R_{k_1,k_2},T_{k_1,k_2},k_{k_1,k_2},2,j})}^2}{{(x^2+1)}^{k_2}} \times (\frac{1}{x^2+1}-\frac{T_{k_1,k_2}}{{(x^2+1)}^2})$. We also need and we can also choose the numbers in such a way that the following hold.
\begin{itemize}
	\item  It holds that $c_2(x) \geq 0$ except for $x$ in a suitable (small) range $-r_{k_1,k_2,R_{k_1,k_2},T_{k_1,k_2},k_{k_1,k_2},1}<x<r_{k_1,k_2,R_{k_1,k_2},T_{k_1,k_2},k_{k_1,k_2},2}$, where $r_{k_1,k_2,R_{k_1,k_2},T_{k_1,k_2},k_{k_1,k_2},1}>0$ and $r_{k_1,k_2,R_{k_1,k_2},T_{k_1,k_2},k_{k_1,k_2},2}>0$ depend on $(k_1,k_2,R_{k_1,k_2},T_{k_1,k_2},k_{k_1,k_2})$.
\item It holds that $c_1(x)<-\frac{1}{2} \times \frac{{\prod}_{j=1}^{k_1-1} {(x-a_{k_1,k_2,R_{k_1,k_2},T_{k_1,k_2},k_{k_1,k_2},1,j})}^2}{{(x^2+1)}^{k_1}}$ for $x$ satisfying $-r_{k_1,k_2,R_{k_1,k_2},T_{k_1,k_2},k_{k_1,k_2},1}<x<r_{k_1,k_2,R_{k_1,k_2},T_{k_1,k_2},k_{k_1,k_2},2}$.
\item  It holds that $c_2(x)>-\frac{1}{4} \times \frac{{\prod}_{j=1}^{k_2-1} {(x-a_{k_1,k_2, R_{k_1,k_2},T_{k_1,k_2},k_{k_1,k_2},2,j})}^2}{{(x^2+1)}^{k_2}}$ for $x$ satisfying $-r_{k_1,k_2,R_{k_1,k_2},T_{k_1,k_2},k_{k_1,k_2},1}<x<r_{k_1,k_2,R_{k_1,k_2},T_{k_1,k_2},k_{k_1,k_2},2}$.
\item It holds that $c_1(x)<c_2(x)$ for $x$ satisfying $-r_{k_1,k_2,R_{k_1,k_2},T_{k_1,k_2},k_{k_1,k_2},1}<x<r_{k_1,k_2,R_{k_1,k_2},T_{k_1,k_2},k_{k_1,k_2},2}$.
\end{itemize}

This completes the proof.
\end{proof}
\subsection{Canonical compactifications.}
\begin{Thm}
\label{thm:3}
Given a $c_1<c_2$-zero-set $(c_1,c_2,D_{c_1,c_2},X_{m,c_1,c_2})$ of SAB-at-Infinity-in-Top such that $c_1(x)$ and $c_2(x)$ both converge to $a_1 \in \mathbb{R}$ and $a_2 \in \mathbb{R}$, as $x$ diverges to $-\infty$ and $+\infty$, respectively. We have a continuous function $f_{S^m,(c_1,c_2,D_{c_1,c_2},X_{m,c_1,c_2})}$ on $S^m$ which is smooth outside the two-point set $\{(0\cdots,1),(0\cdots, -1)\} \subset S^m$ canonically, and this also gives a pre-digraph\\ $(R_{f_{S^m,(c_1,c_2,D_{c_1,c_2},X_{m,c_1,c_2})}},\\
q_{f_{S^m,(c_1,c_2,D_{c_1,c_2},X_{m,c_1,c_2})}}(S(f_{S^m,(c_1,c_2,D_{c_1,c_2},X_{m,c_1,c_2})} {\mid}_{S^m-\{(0\cdots,1),(0\cdots, -1)\}}))\\
\bigcup q_{f_{S^m,(c_1,c_2,D_{c_1,c_2},X_{m,c_1,c_2})}}(\{(0\cdots,1),(0\cdots, -1)\}),\bar{f_{S^m,(c_1,c_2,D_{c_1,c_2},X_{m,c_1,c_2})}})$.
\end{Thm}
\begin{proof}
Due to the construction, there exists some diffeomorphism ${\phi}_{S^m,(c_1,c_2,D_{c_1,c_2},X_{m,c_1,c_2})}$ from $S^m-\{(0 \cdots,1),(0 \cdots, -1)\}$ onto $X_{m,c_1,c_2}$ and a desired function $f_{S^m,(c_1,c_2,D_{c_1,c_2},X_{m,c_1,c_2})}$ can be defined as follows.
\begin{itemize}
\item The restriction to $S^m-\{(0 \cdots,1),(0 \cdots, -1)\}$ is ${\pi}_{m+1,1} {\mid}_{X_{m,c_1,c_2}} \circ {\phi}_{S^m,(c_1,c_2,D_{c_1,c_2},X_{m,c_1,c_2})}$.
\item At $(0 \cdots,-1)$ and $(0 \cdots,1)$, the values are $a_1$ and $a_2$, respectively.
\end{itemize}
For {\it Reeb graphs} of continuous functions on (smooth) manifolds with no boundary such that the restrictions to subsets whose complementary sets in the manifolds are of Lebesgue measure $0$ are smooth, see also the pioneering preprint \cite{kitazawa5} and related preprint \cite{kitazawa10}. We do not assume non-trivial knowledge on the preprints, in the present paper, of course.

We go back to exposition of the proof. We have a desired pre-digraph by the construction. For Reeb spaces of continuous functions on (smooth) closed manifolds, see also \cite{gelbukh1}. They are metrizable Peano continua.
From this, we can also show that\\
$(R_{f_{S^m,(c_1,c_2,D_{c_1,c_2},X_{m,c_1,c_2})}},\\
q_{f_{S^m,(c_1,c_2,D_{c_1,c_2},X_{m,c_1,c_2})}}(S(f_{S^m,(c_1,c_2,D_{c_1,c_2},X_{m,c_1,c_2})} {\mid}_{S^m-\{(0\cdots,1),(0\cdots, -1)\}}))\\
\bigcup q_{f_{S^m,(c_1,c_2,D_{c_1,c_2},X_{m,c_1,c_2})}}(\{(0\cdots,1),(0\cdots, -1)\}),\bar{f_{S^m,(c_1,c_2,D_{c_1,c_2},X_{m,c_1,c_2})}})$ is a pre-digraph, immediately. 
\end{proof}
\begin{Rem}
\label{rem:1}
In Theorem \ref{thm:3}, we drop the condition that $c_1(x)$ and $c_2(x)$ both converge to $a_1 \in \mathbb{R}$ and $a_2 \in \mathbb{R}$, as $x$ diverges to $-\infty$ and $+\infty$, respectively. In other words, they may diverge. In this general case, we have a continuous map to $S^1$. For this, we can define its {\it Reeb space} similarly and the Reeb space enjoys similar topological properties and combinatorial ones, which is shown by the presented arguments and facts, similarly. The proof is left to readers.
\end{Rem}
\begin{Thm}
\label{thm:4}
We consider the function $f_{S^m,(c_1,c_2,D_{c_1,c_2},X_{m,c_1,c_2})}$ defined in {\rm (}the proof of{\rm )} Theorem \ref{thm:3}. If in Theorem \ref{thm:1}, for the $c_1<c_2$-zero-set $(c_1,c_2,D_{c_1,c_2},X_{m,c_1,c_2})$ of SAB-at-Infinity-in-Top such that $c_{1}(x)$ and $c_{2}(x)$ both converge to $a_1 \in \mathbb{R}$ and $a_2 \in \mathbb{R}$, as $x$ diverges to $-\infty$ and $+\infty$, respectively, either of {\rm (}\ref{thm:1.3.2}, \ref{thm:1.3.4}, \ref{thm:1.3.5}, \ref{thm:1.3.6}{\rm )} holds, then\\ ${\rm GDNF}(R_{f_{S^m,(c_1,c_2,D_{c_1,c_2},X_{m,c_1,c_2})}},\\
q_{f_{S^m,(c_1,c_2,D_{c_1,c_2},X_{m,c_1,c_2})}}(S(f_{S^m,(c_1,c_2,D_{c_1,c_2},X_{m,c_1,c_2})} {\mid}_{S^m-\{(0\cdots,1),(0\cdots, -1)\}}))\\
 \bigcup q_{f_{S^m,(c_1,c_2,D_{c_1,c_2},X_{m,c_1,c_2})}}(\{(0\cdots,1),(0\cdots, -1)\}),\bar{f_{S^m,(c_1,c_2,D_{c_1,c_2},X_{m,c_1,c_2})}})$ is isomorphic to ${\rm GDNF}(R_{{\pi}_{m+1,1} {\mid}_{X_{m,c_1,c_2}}},q_{{\pi}_{m+1,1} {\mid}_{X_{m,c_1,c_2}}}(S({\pi}_{m+1,1} {\mid}_{X_{m,c_1,c_2}})),\bar{{\pi}_{m+1,1} {\mid}_{X_{m,c_1,c_2}}})$.
\end{Thm}
\begin{proof}
This immediately follows from our definitions and formulation. Remember that for the case of Theorem \ref{thm:1} (\ref{thm:1.3.2}), either $c_1$ or $c_2$ must be constant.

\end{proof}

We present a kind of counterexamples related to Theorem \ref{thm:4}.

Let $c_{A,1,1}:=\frac{1}{2(x^2+1)}-\frac{\sin (e^{x^2})}{2{(x^2+1)}^2}>0$, $c_{A,1,2}:=\frac{1}{x^2+1}$ and we have $0<c_{A,1,1}(x)<c_{A,1,2}(x)$ for any $x \in \mathbb{R}$. This is a case for Theorem \ref{thm:1} (\ref{thm:1.3.1})
 and $c_{A,1,1}(x)$ and $c_{A,1,2}(x)$ converge to $0$ at each infinity $\pm \infty$. ${\rm GDNF}(R_{f_{S^m,(c_{A,1,1},c_{A,1,2},D_{c_{A,1,1},c_{A,1,2}},X_{m,c_{A,1,1},c_{A,1,2}})}},\\
q_{f_{S^m,(c_{A,1,1},c_{A,1,2},D_{c_{A,1,1},c_{A,1,2}},X_{m,c_{A,1,1},c_{A,1,2}})}} \\(S(f_{S^m,(c_{A,1,1},c_{A,1,2},D_{c_{A,1,1},c_{A,1,2}},X_{m,c_{A,1,1},c_{A,1,2}})} {\mid}_{S^m-\{(0\cdots,1),(0\cdots, -1)\}}))\\
 \bigcup q_{f_{S^m,(c_{A,1,1},c_{A,1,2},D_{c_{A,1,1},c_{A,1,2}},X_{m,c_{A,1,1},c_{A,1,2}})}}(\{(0\cdots,1),(0\cdots, -1)\}),
\\
\bar{f_{S^m,(c_{A,1,1},c_{A,1,2},D_{c_{A,1,1},c_{A,1,2}},X_{m,c_{A,1,1},c_{A,1,2}})}})$ is a $1$-dimensional connected cell complex with exactly one vertex and exactly two edges whose closures taken in the cell complex are homeomorphic to $\{t \mid t>0\}$ and which are both oriented as edges departing from the vertex.

Let $c_{A,2,1}:=-\frac{1}{2(x^2+1)}<0$, $c_{A,2,2}:=\frac{e^{-x^2}{\sin} x}{3(x^2+1)}$ and we have $c_{A,2,1}(x)<0$ and $c_{A,2,1}(x)<c_{A,2,2}(x)$ for any $x \in \mathbb{R}$. This is a case for Theorem \ref{thm:1} (\ref{thm:1.3.1})
 and $c_{A,2,1}(x)$ and $c_{A,2,2}(x)$ converge to $0$ at each infinity $\pm \infty$.
The sets $\{x \mid c_{A,2,2}(x)>0\}$, $\{x \mid c_{A,2,2}(x)<0\}$ and $\{x \mid {c_{A,2,2}}^{\prime}(x)=0\}$ are also unbounded: note that ${c_{A,2,2}}^{\prime}(x)$ means the 1st derivative. 

 ${\rm GDNF}(R_{f_{S^m,(c_{A,2,1},c_{A,2,2},D_{c_{A,2,1},c_{A,2,2}},X_{m,c_{A,2,1},c_{A,2,2}})}},\\
q_{f_{S^m,(c_{A,2,1},c_{A,2,2},D_{c_{A,2,1},c_{A,2,2}},X_{m,c_{A,2,1},c_{A,2,2}})}} \\(S(f_{S^m,(c_{A,2,1},c_{A,2,2},D_{c_{A,2,1},c_{A,2,2}},X_{m,c_{A,2,1},c_{A,2,2}})} {\mid}_{S^m-\{(0\cdots,1),(0\cdots, -1)\}}))\\
 \bigcup q_{f_{S^m,(c_{A,2,1},c_{A,2,2},D_{c_{A,2,1},c_{A,2,2}},X_{m,c_{A,2,1},c_{A,2,2}})}}(\{(0\cdots,1),(0\cdots, -1)\}),
\\
\bar{f_{S^m,(c_{A,2,1},c_{A,2,2},D_{c_{A,2,1},c_{A,2,2}},X_{m,c_{A,2,1},c_{A,2,2}})}})$ is a $1$-dimensional connected cell complex with exactly one vertex and exactly two edges whose closures taken in the cell complex are homeomorphic to $S^1$.


This is summarized as follows.
\begin{Thm}
\label{thm:5}
For the function $f_{S^m,(c_1,c_2,D_{c_1,c_2},X_{m,c_1,c_2})}$ defined in {\rm (}the proof of{\rm )} Theorem \ref{thm:3}, we can have a case for Theorem \ref{thm:1} {\rm (}\ref{thm:1.3.1}{\rm )} such that $c_1(x)$ and $c_2(x)$ both converge to $0$, as $x$ diverges to $-\infty$ and $+\infty$ and that  ${\rm GDNF}(R_{f_{S^m,(c_{1},c_{2},D_{c_{1},c_{2}},X_{m,c_{1},c_{2}})}},\\
q_{f_{S^m,(c_{1},c_{2},D_{c_{1},c_{2}},X_{m,c_{1},c_{2}})}} \\(S(f_{S^m,(c_{1},c_{2},D_{c_{1},c_{2}},X_{m,c_{1},c_{2}})} {\mid}_{S^m-\{(0\cdots,1),(0\cdots, -1)\}}))\\
 \bigcup q_{f_{S^m,(c_{1},c_{2},D_{c_{1},c_{2}},X_{m,c_{1},c_{2}})}}(\{(0\cdots,1),(0\cdots, -1)\}),
\\
\bar{f_{S^m,(c_{1},c_{2},D_{c_{1},c_{2}},X_{m,c_{1},c_{2}})}})$ \\
is as follows, for example.
\begin{enumerate}
\item A one-point set.
\item A $1$-dimensional connected cell complex with exactly one vertex and exactly two edges whose closures taken in the cell complex are homeomorphic to $\{t \mid t>0\}$ and which are both oriented as edges departing from the vertex.
\item A $1$-dimensional connected cell complex with exactly one vertex and exactly two edges whose closures taken in the cell complex are homeomorphic to $\{t \mid t>0\}$ and which are both oriented as edges entering the vertex.
\item A $1$-dimensional connected cell complex with exactly one vertex and exactly two edges whose closures taken in the cell complex are homeomorphic to $S^1$.
\end{enumerate}
\end{Thm}
\begin{proof}
We only explain the first case and the third case to complete the proof. The first case is obtained as in STEP 1-3-1 of our proof of Theorem \ref{thm:1}. The third case is obtained by changing the signs of the functions in the second case.
\end{proof}

Related further precise studies on these compactifications are left to us as a part of our future studies.
\section{Conflict of interest and Data availability.}
 The author is a researcher at Osaka Central Advanced Mathematical Institute (OCAMI researcher). This is supported by MEXT Promotion of Distinctive Joint Research Center Program JPMXP0723833165. He thanks the people for the hospitality, where he is not employed by the institute or the projects. 
 
No data other than the present file is generated, related to the present study. We do not assume non-trivial arguments in preprints which are still unpublished. We may refer to these preprints to some extent.

\end{document}